\documentclass[english,12pt]{article}
\usepackage{amssymb}

\usepackage{bm}
\usepackage{amsmath}
\usepackage{amsfonts}
\usepackage{mathrsfs}
\usepackage{graphicx}
\usepackage{tikz}
\usetikzlibrary{arrows,shapes,positioning}
\usetikzlibrary{decorations.markings}
\tikzstyle arrowstyle=[scale=1.5]
\tikzstyle directed=[postaction={decorate,decoration={markings, mark=at position .9 with {\arrow[arrowstyle]{stealth}}}}]
\tikzstyle reverse directed=[postaction={decorate,decoration={markings, mark=at position .9 with {\arrowreversed[arrowstyle]{stealth};}}}]
\usepackage{}

\usetikzlibrary{shadows}

\newtheorem{Theorem}{Theorem}[section]
\newtheorem{Corollary}{Corollary}[section]
\newtheorem{Definition}{Definition}[section]
\newtheorem{Conjecture}{Conjecture}[section]

\newtheorem{Lemma}{Lemma}[section]

\def\emptyset{\mbox{{\rm \O}}}

\newenvironment{proof}{
\noindent {\bf Proof.}\rm}%
{\mbox{}\hfill\rule{0.5em}{0.809em}\par}
\usepackage{calc}
\usepackage{tikz}
\usetikzlibrary{decorations.markings}
\tikzstyle{vertex}=[circle, draw, inner sep=0pt, minimum size=6pt]
\newcommand{\vertex}{\node[vertex]}

\begin{document}
\title{\LARGE{\textbf{The generalized Tur\'{a}n number of long cycles in graphs and bipartite graphs
\thanks{This work was supported by the
the Natural Science Foundation of Xinjiang (No. 2020D04046) and the
National Natural Science
Foundation of China (No. 12261016, No. 11901498, No. 12261085).}}}}
\author{{Changchang Dong$^{a}$, Mei Lu$^{a}$, Jixiang Meng$^{b}$, Bo Ning$^{c}$\footnote{Corresponding
author.
E-mail: bo.ning@nankai.edu.cn}}
\\
\small  a. Department of Mathematical Sciences, Tsinghua University, Beijing 100084, China  \\
\small  b. College of Mathematics and System Sciences, Xinjiang University, Urumqi 830046, China  \\
\small  c. College of Computer Science, Nankai University, Tianjin 300071, China  \\
}

\date{}
\maketitle

\noindent {\small {\bf Abstract}\\
Given a graph $T$ and a family of graphs $\mathcal{F}$, the maximum number of copies of $T$ in an $\mathcal{F}$-free
graph on $n$ vertices is called the generalized Tur\'{a}n number, denoted by $ex(n, T , \mathcal{F})$. When $T= K_2$, it
reduces to the classical Tur\'{a}n number $ex(n, \mathcal{F})$.
 Let
$ex_{bip}(b,n, T , \mathcal{F})$ be the maximum number of copies of $T$ in
an $\mathcal{F}$-free bipartite graph with
 two parts of sizes $b$ and $n$, respectively.
Let $P_k$ be the path on $k$ vertices, $\mathcal{C}_{\ge k}$ be
the family of all cycles with length at least $k$ and $M_k$ be a matching with $k$
edges.
 In this article, we determine $ex_{bip}(b,n, K_{s,t}, \mathcal{C}_{\ge 2n-2k})$ exactly in a connected
bipartite graph $G$ with minimum degree $\delta(G) \geq r\ge 1$, for $b\ge n\ge 2k+2r$ and $k\in \mathbb{Z}$,
which generalizes a theorem of Moon and Moser, a theorem of Jackson and gives an affirmative evidence supporting a conjecture of Adamus and Adamus.
As corollaries of our main result, we determine $ex_{bip}(b,n, K_{s,t}, P_{2n-2k})$ and $ex_{bip}(b,n, K_{s,t}, M_{n-k})$ exactly in a connected
bipartite graph $G$ with minimum degree $\delta(G) \geq r\ge 1$, which generalizes a theorem of Wang.
Moreover, we determine $ex(n, K_{s,t}, \mathcal{C}_{\ge k})$ and $ex(n, K_{s,t}, P_{k})$ respectively in a connected
 graph $G$ with minimum degree $\delta(G) \geq r\ge 1$, which generalizes a theorem of Lu, Yuan and Zhang.
}
\\



\section{Introduction}
Throughout this article we only consider finite simple graphs.
For an integer $n$, we define $[n]=\{1,2,\ldots ,n\}$.
For a graph $G$, we denote by $e(G)$ the edge number of $G$ and by $\delta(G)$ the minimum degree of $G$.
Let $d_G(v)$ and $N_G(v)$ denote, respectively, the \emph{degree} and
 \emph{neighbourhood} of a vertex $v\in V(G)$.
For a subgraph
$H$ of $G$, we denote
  $d_H (v) = |N_G(v)
\cap V (H)|$ and
use $G- H$ to denote the graph obtained from $G$ by deleting the vertices of
$H$ and the edges incident with at least one vertex in $H$.
The \emph{length} of a cycle or a path is the number of its edges.
We use $P_k,\mathcal{C}_{\ge k}$ and $M_k$ to denote the path on $k$
vertices, the family of all cycles with length at least $k$ and a matching with $k$ edges, respectively.
For positive integers $s$ and $t$, we use $K_s$ to denote the \emph{complete graph} on
$s$ vertices and use
 $K_{s,t}$ to denote the \emph{complete bipartite graph}
with two parts of size $s$ and $t$, respectively.
 A graph $G$ is called
\emph{Hamiltonian} if it contains a Hamilton cycle, i.e., a cycle that includes every
vertex of $G$.
A bipartite graph with the bipartition $\{X,Y\}$ is called \emph{balanced} if $|X|=|Y|$.

Let $T$ be a graph and $\mathcal{F}$ be a family of graphs.
We use $N(T, G)$ to denote the number of (not necessarily induced) copies
of $T$ in $G$
and say that $G$ is $\mathcal{F}$-free if there is no copy of any member of $\mathcal{F}$ in $G$.
The \emph{generalized Tur\'{a}n number}, denoted by $ex(n, T , \mathcal{F})$,
is
the maximum number of copies of $T$ in an $\mathcal{F}$-free
graph on $n$ vertices.  When $\mathcal{F}=\{F\}$, we write $ex(n, T , F)$ instead of $ex(n, T , \{F\})$.
When $T= K_2$, it
reduces to the classical \emph{Tur\'{a}n number} $ex(n, \mathcal{F})$.
 Let
$ex_{bip}(b,n, T , \mathcal{F})$ be the maximum number of copies of $T$ in
an $\mathcal{F}$-free bipartite graph with
 two parts of sizes $b$ and $n$, respectively.
 The generalized Tur\'{a}n number has received a lot of attention recently, see \cite{AS16,GG20,GMV19,GP19,GS20,GL12,MQ20}.

  Erd\H{o}s and Gallai \cite{EG59} first studied
the maximum number of edges in $P_k$-free graphs and $\mathcal{C}_{\ge k}$-free graphs on $n$ vertices and
characterized the extremal graphs for some values of $n$.
Kopylov \cite{Ko77} extended Erd\H{o}s and Gallai's results to 2-connected graphs by
determining $ex(n, \mathcal{C}_{\ge k})$ exactly in 2-connected graphs and determining $ex(n, P_{ k})$ exactly in connected graphs.
Luo \cite{Luo17} determined $ex(n, K_s,\mathcal{C}_{\ge k})$ exactly in 2-connected graphs
and determined $ex(n, K_s,P_ k)$ exactly in connected graphs, which generalized Kopylov's results to $K_s$.
In \cite{NP20},  Ning and Peng presented an extension of Luo's results by imposing minimum degree as a new parameter,
which was in the spirit of Kopylov's remark.
Another extension of Kopylov's results to $K_{s,t}$ was proposed by Lu, Yuan and Zhang \cite{LYZ21} in 2021s,
who determined $ex(n, K_{s,t},\mathcal{C}_{\ge k})$ exactly in 2-connected graphs and determining $ex(n, K_{s,t},P_{ k})$ exactly in connected graphs.
Inspired by Ning-Peng's results and Lu-Yuan-Zhang's results,
we present a generalization of Lu-Yuan-Zhang's results by imposing minimum degree as a new parameter in this article, and the results will be shown in Section 2.

%

In extremal graph theory, a natural problem on Hamilton cycles is: How many edges can guarantee the existence of a Hamilton cycle in a graph on $n$ vertices.
Ore \cite{Ore61} showed that the condition $e(G) \geq\left(\begin{array}{c}n-1 \\ 2\end{array}\right)+2$ is the answer.

\begin{Theorem} (Ore \cite{Ore61}) \label{Ore61}
Let $G$ be a graph on $n$ vertices. If
$$
e(G)>\left(\begin{array}{c}
n-1 \\
2
\end{array}\right)+1
$$
then $G$ contains a Hamilton cycle.
\end{Theorem}

In 1962s, Erd\H{o}s \cite{Er62}
given an exension of
Ore's theorem by adding a bound on the minimum degree as a new parameter.

\begin{Theorem} (Erd\H{o}s \cite{Er62})\label{Er62}
Let $G$ be a graph on $n$ vertices. If $\delta(G) \geq r$, where $1 \leq r \leq$ $\frac{n-1}{2}$, and
$$
e(G)>\max \left\{\left(\begin{array}{c}
n-r \\
2
\end{array}\right)+r^2,\left(\begin{array}{c}
n-\left\lfloor\frac{n-1}{2}\right\rfloor \\
2
\end{array}\right)+\left\lfloor\frac{n-1}{2}\right\rfloor^2\right\}
$$
then $G$ contains a Hamilton cycle.
\end{Theorem}


Motivated by Erd\H{o}s' work, Moon and Moser \cite{MM63} presented corresponding result for balanced bipartite graphs.
\begin{Theorem} (Moon and Moser \cite{MM63})\label{MM63}
Let $G$ be a balanced bipartite graph on $2 n$ vertices, with minimum degree $\delta(G) \geq r$, where $1 \leq r \leq \frac{n }{2}$. If
$$
e(G)>n(n-r)+r^2,
$$
then $G$ contains a Hamilton cycle.
\end{Theorem}


In 2009s, Adamus and Adamus \cite{AA09} wanted to generalize the above criteria to long cycles by
asking how many edges are needed in a balanced bipartite graph on $2n$ vertices, to ensure
the existence of a cycle of length exactly $2n - 2k$?
In particular, they settled the case of $k=1$.

\begin{Conjecture} (Adamus and Adamus \cite{AA09})\label{AA09}
Let $G$ be a balanced bipartite graph of order $2n$ and minimal degree $\delta (G)\ge r \ge 1$,
where $n \ge 2k + 2r$ and $k\in \mathbb{Z}$. If
$$e(G) > n(n - k - r) + r(k + r)$$
then $G$ contains a cycle of length $2n - 2k$.
\end{Conjecture}

Jackson \cite{Ja85} determined the minimum number of edges in a bipartite graph without the parameter minimum degree
to ensure the existence of a cycle of length at least $2n-2k$.

\begin{Theorem} (Jackson \cite{Ja85})\label{Ja85}
 Let $t$ be an integer and $G=(X, Y ; E)$ be a bipartite graph. Suppose $|X|=n,|Y|=b$, where $b \geq n \geq n-k \geq 2$. If
$$
e(G)> \begin{cases}(b-1)(n-k-1)+n, & n \leq 2(n-k)-2 \\ (b-n+2k+3)(n-k-1), & n \geq 2 (n-k)-2\end{cases}
$$
then $G$ contains a cycle of length at least $2n-2k$.
\end{Theorem}

From Theorem \ref{Ja85}, we have
 $\operatorname{ex}_{bip}(b, n, \mathcal{C}_{\ge 2n-2k})=(n-k-1)b+k+1$ if $b \geq n \geq n-k \geq \frac{n}{2}+1$.
Li and the forth author \cite{LN21} improved Theorem \ref{Ja85}.

\begin{Theorem} (Li and Ning \cite{LN21})\label{LN21}
$\operatorname{ex}_{bip}(b, n, C_{2n-2k})=(n-k-1)b+k+1$ if $b \geq n \geq n-k \geq \frac{n}{2}+1$.
 \end{Theorem}

 Therefore
 if
 Conjecture \ref{AA09} is true, then it can be seen as
 a generalization of Theorem \ref{Ja85} and Theorem \ref{LN21} by imposing minimum degree as a new parameter.

In 2021s, Wang \label{Wang20} determined the exact value of the generalized Tur\'{a}n number of matchings.
He proved the following theorem.

\begin{Theorem} (Wang \cite{Wang20})\label{Wang20}
For any $s, t \geq 1$ and $n \geq n-k-1$, we have
$$
e x_{b i p}\left(n,n, K_{s, t}, M_{n-k}\right)= \begin{cases}\left(\begin{array}{c}
n-k-1 \\
s
\end{array}\right)\left(\begin{array}{c}
n \\
t
\end{array}\right)+\left(\begin{array}{c}
n-k-1 \\
t
\end{array}\right)\left(\begin{array}{c}
n \\
s
\end{array}\right), & s \neq t, \\
\left(\begin{array}{c}
n-k-1 \\
s
\end{array}\right)\left(\begin{array}{c}
n \\
s
\end{array}\right), & s=t\end{cases}
$$
\end{Theorem}

To our knowledge, there are no further references on
Conjecture \ref{AA09}. In this article, we give a first step towards confirming Conjecture \ref{AA09}, which is an affirmative evidence supporting
this conjecture.
In particular,
we determine $ex_{bip}(b,n, K_{s,t}, \mathcal{C}_{\ge 2n-2k})$ exactly in a connected
bipartite graph $G$ with minimum degree $\delta(G) \geq r\ge 1$, for $b\ge n\ge 2k+2r$ and $k\in \mathbb{Z}$,
which generalizes Theorem \ref{MM63} and Theorem \ref{Ja85} for $b \geq n \geq n-k \geq \frac{n}{2}+1$.
As corollaries of our main result, we determine $ex_{bip}(b,n, K_{s,t}, P_{2n-2k})$ and $ex_{bip}(b,n, K_{s,t}, M_{n-k})$ exactly in a connected
bipartite graph $G$ with minimum degree $\delta(G) \geq r\ge 1$, which generalizes Theorem \ref{Wang20}.
Moreover, we determine $ex(n, K_{s,t}, \mathcal{C}_{\ge k})$ in a 2-connected
 graph $G$ with minimum degree $\delta(G) \geq r\ge 1$, which generalizes a theorem of Lu, Yuan and Zhang.
Our results as shown in Section 2, and the proofs as demonstrate in Section 3.
The last section is devoted to some concluding remarks.
\section{Main Results}

For $n \geqslant 2k+2r,1\le a\le \lfloor\frac{n-k}{2}\rfloor $ and $k\in \mathbb{Z}$, let
$$
f_{s, t}(b,n,n-k,a)= \left(\begin{array}{c}
b \\
s
\end{array}\right)\left(\begin{array}{c}
n-k-a \\
t
\end{array}\right)+\left(\begin{array}{c}
a \\
s
\end{array}\right)\left(\begin{array}{c}
n \\
t
\end{array}\right)-
\left(\begin{array}{c}
a \\
s
\end{array}\right)
\left(\begin{array}{c}
n-k-a \\
t
\end{array}\right),
$$
It can be checked that $f_{s, t}(b,n,n-k,a)$ is a convex function.
$$
\begin{tikzpicture}
[x=1.1cm, y=0.7cm, every edge/.style={draw, postaction={decorate,decoration={markings,mark=at position 0.6 with {\arrow{>}}}}}]
       \draw[] (-1,2.5) ellipse (1.5 and 0.5);\draw[] (3,2.5) ellipse (1.8 and 0.5);
       \draw[] (-1,-0.5) ellipse (1.3 and 0.5);\draw[] (3,-0.5) ellipse (2 and 0.5);

       \draw [fill=black] (-2,2.5)  circle (0.08cm); \draw [fill=black] (0,2.5)  circle (0.08cm);
       \draw [fill=black] (2,2.5)  circle (0.08cm);  \draw [fill=black] (4,2.5)  circle (0.08cm);
       \draw [fill=black] (-2,-0.5)  circle (0.08cm); \draw [fill=black] (0,-0.5)  circle (0.08cm);
       \draw [fill=black] (2,-0.5)  circle (0.08cm);  \draw [fill=black] (4,-0.5)  circle (0.08cm);
       \node at (-3,2.5){$X$};\node at (-3,-0.5){$Y$};
        \node at (-1,3.3){$A$};\node at (3,3.3){$B$};
           \node at (-1,-1.3){$C$};\node at (3,-1.3){$D$};
           \node at (-1,2.5){$\cdots$};\node at (3,2.5){$\cdots$};
           \node at (-1,-0.5){$\cdots$};\node at (3,-0.5){$\cdots$};
           \draw[] (-2,2.5)--(-2,-0.5);
            \draw[] (0,2.5)--(0,-0.5);\draw[] (-2,2.5)--(0,-0.5);
            \draw[] (0,2.5)--(-2,-0.5);

            \draw[] (4,2.5)--(0,-0.5);
            \draw[] (4,2.5)--(-2,-0.5);\draw[] (2,2.5)--(0,-0.5);
            \draw[] (2,2.5)--(-2,-0.5);

         \draw[] (4,2.5)--(4,-0.5);
            \draw[] (4,2.5)--(2,-0.5);\draw[] (2,2.5)--(4,-0.5);
            \draw[] (2,2.5)--(2,-0.5);

            \node at (1,-2){Figure 1. $F_{b,n,n-k,a}$.};
\end{tikzpicture}
$$

Let $F_{b,n,n-k,a}$ be a bipartite graph, with colour classes $X$ and $Y$, $|X| =n\le |Y | = b$, where
$X = A\cup B, Y = C\cup D, |A| = k + a, |B| = n -k - a, |C| = a$, and $|D| = b - a$. Moreover, assume
that $N_{F_{b,n,n-k,a}}(x) = C$ for all $x \in A$, and $N_{F_{b,n,n-k,a}}
(x) = Y$ for all $x\in B$. See Figure 1.
Note that if $n\ge 2k+2a$, then $\delta(G_1)=a\ge 1$ and $$N(K_{s,t},F_{b,n,n-k,a})= \begin{cases}
f_{s, t}(b,n,n-k,a), & s = t, \\
f_{s, t}(b,n,n-k,a)+ f_{t, s}(b,n,n-k,a), & s \not= t, \end{cases}$$
but $F_{b,n,n-k,a}$ does not contain a cycle of length $2n-2k$ or more.

\begin{Theorem}\label{CB}
Let $G$ be a connected bipartite graph with bipartition $(X, Y)$.
Suppose $|X|=n\le |Y|=b,h=\lfloor\frac{n-k}{2}\rfloor$ and $\delta(G) \geq r\ge 1$, where $n\ge 2k+2r$ and $k\in \mathbb{Z}$.
$$
N(K_{s,t},G)> \begin{cases}
\max \{f_{s, t}(b,n,n-k,r),f_{s, s}(b,n,n-k,h)\}, & s = t, \\
\max \{f_{s, t}(b,n,n-k,r)+ f_{t, s}(b,n,n-k,r),\\f_{s, t}(b,n,n-k,h)+
f_{t, s}(b,n,n-k,h)\}, & s \not= t. \end{cases}
$$
then $G$ contains a cycle of length at least $2n-2k$.
\end{Theorem}

This theorem is sharp with the extremal graphs $F_{b,n,n-k,r}$ and $F_{b,n,n-k,h}$.

\begin{Theorem}\label{PB}
Let $G$ be a connected bipartite graph with bipartition $(X, Y)$.
Suppose $|X|=n\le |Y|=b,h=\lfloor\frac{n-k-1}{2}\rfloor$ and $\delta(G) \geq r\ge 1$, where $n\ge 2k+2r$ and $k\in \mathbb{Z}$. If
$$
N(K_{s,t},G)> \begin{cases}
\max \{f_{s, t}(b,n,n-k-1,r),f_{s, s}(b,n,n-k-1,h)\}, & s = t, \\
\max \{f_{s, t}(b,n,n-k-1,r)+ f_{t, s}(b,n,n-k-1,r),\\f_{s, t}(b,n,n-k-1,h)+f_{t, s}(b,n,n-k-1,h)\}, & s \not= t. \end{cases}
$$
then $G$ contains a path on $2n-2k$ vertices.
\end{Theorem}

The following result is a direct corollary.

\begin{Corollary}\label{MB}
Let $G$ be a connected bipartite graph with bipartition $(X, Y)$.
Suppose $|X|=n\le |Y|=b,h=\lfloor\frac{n-k-1}{2}\rfloor$ and $\delta(G) \geq r\ge 1$, where $n\ge 2k+2r$ and $k\in \mathbb{Z}$. If
$$
N(K_{s,t},G)> \begin{cases}
\max \{f_{s, t}(b,n,n-k-1,r),f_{s, s}(b,n,n-k-1,h)\}, & s = t, \\
\max \{f_{s, t}(b,n,n-k-1,r)+ f_{t, s}(b,n,n-k-1,r),\\f_{s, t}(b,n,n-k-1,h)+f_{t, s}(b,n,n-k-1,h)\}, & s \not= t. \end{cases}
$$
then $G$ contains a matching with $n-k$
edges.
\end{Corollary}

Theorem \ref{PB} and Corollary \ref{MB} are sharp with the extremal graphs $F_{b,n,n-k-1,r}$ and $F_{b,n,n-k-1,h}$.

For $n \geqslant k \geqslant 4$ and $1\le r <\frac{k}{2} $, let
$$
g_{s, t}(n,k,a)= \begin{cases}\sum\limits_{i=1}\limits^{n-k+a}\left(\begin{array}{c}
a \\
s
\end{array}\right)\left(\begin{array}{c}
n-s-i \\
s-1
\end{array}\right)+\frac{1}{2}\left(\begin{array}{c}
k-a \\
2 s
\end{array}\right)\left(\begin{array}{c}
2 s \\
s
\end{array}\right),&s=t; \\
\sum\limits_{i=1}\limits^{n-k+a}\left(\left(\begin{array}{c}
a \\
s
\end{array}\right)\left(\begin{array}{c}
n-s-i \\
t-1
\end{array}\right)+\left(\begin{array}{c}
a \\
t
\end{array}\right)\left(\begin{array}{c}
n-t-i \\
s-1
\end{array}\right)\right)+\\\left(\begin{array}{c}
k-a \\
s+t
\end{array}\right)\left(\begin{array}{c}
s+t \\
s
\end{array}\right),&s \neq t.\end{cases}
$$

It can be checked that $g_{s, t}(n, k, a)$ is a convex function.

For $n \geqslant k \geqslant 4$ and $k / 2>a \geqslant 1$, define the $n$-vertex graph $H_{n, k, a}$ as follows. The vertex set of $H_{n, k, a}$ is partitioned into three sets $A, B, C$ such that $|A|=a,|B|=n-k+a$ and $|C|=k-2 a$ and the edge set of $H_{n, k, a}$ consists of all edges between $A$ and $B$ together with all edges in $A \cup C$
(see Figure 2).
Note that when $a \geqslant 2, H_{n, k, a}$ is 2-connected, has no cycle of length $k$ or more.

$$
\begin{tikzpicture}
[x=1cm, y=0.5cm, every edge/.style={draw, postaction={decorate,decoration={markings,mark=at position 0.6 with {\arrow{>}}}}}]
       \draw[] (-1,2.5) ellipse (1.3 and 4);\draw[fill=gray] (3,2.5) ellipse (1.2 and 4);\draw[fill=gray] (2.4,2.5) ellipse (0.5 and 2);
        \node at (-1,7){$B$}; \node at (3,7){$K_{k-a}$};\node at (2.4,4.9){$A$};\node at (3.4,4.9){$C$};

        \vertex (1) at (-1,5.4)[label=left:$b_1$,fill=black] {};
        \vertex (2) at (-1,4.4) [label=left:$b_2$,fill=black]{};
        \node at (-1,3.1){$\vdots$};
        \node at (-1,2.3){$\vdots$};
        \vertex (3) at (-1,0.8)[fill=black] {};
        \vertex (4) at (-1,-0.2)[fill=black] {};
         \node at (-1,-0.5){$b_{n-k+a}$};
         \vertex (5) at (-1,0.8)[fill=black] {};
        \vertex (1') at (2.4,3.8)[fill=black] {}; \node at (2.4,3.15){$\vdots$};\node at (2.4,2.35){$\vdots$};  \vertex (2') at (2.4,1.4)[fill=black] {};
        \draw[] (2.4,3.8)--(-1,5.4);\draw[] (2.4,3.8)--(-1,4.4);\draw[] (2.4,3.8)--(-1,0.8);\draw[] (2.4,3.8)--(-1,-0.2);
         \draw[] (2.4,1.4)--(-1,5.4);\draw[] (2.4,1.4)--(-1,4.4);\draw[] (2.4,1.4)--(-1,0.8);\draw[] (2.4,1.4)--(-1,-0.2);
       \node at (0.9,-2.6){Figure 2. $H_{n, k, a}$.};
\end{tikzpicture}
$$

For $B \subset V\left(H_{n, k, a}\right)$, let $B=\left\{b_1, b_2, \cdots, b_{n-k+a}\right\}$.
Note that for $i \in[n-k+a]$, the number of copies of $K_{s, t}$ containing $b_i$ and not
containing $b_1, \cdots, b_{i-1}$ is $\left(\begin{array}{c}a \\ s\end{array}\right)\left(\begin{array}{c}n-s-i \\ s-1\end{array}\right)$ when $s=t$ and $\left(\begin{array}{c}a \\ s\end{array}\right)\left(\begin{array}{c}n-s-i \\ t-1\end{array}\right)+\left(\begin{array}{c}a \\ t\end{array}\right)\left(\begin{array}{c}n-t-i \\ s-1\end{array}\right)$ when $s \neq t$. Hence, the number of copies of $K_{s, t}$ in $H_{n, k, a}$ is $g_{s, t}(n, k, a)$.

\begin{Theorem}\label{C}

 Let $G$ be a 2-connected graph on n vertices, with $\delta(G) \geq r\ge 2$, where
$n \geqslant k \geqslant 5$ and $h=\lfloor\frac{k-1}{2}\rfloor$. If
$$
N(K_{s,t},G)>\max \left\{g_{s, t}(n,k,r), g_{s, t}(n,k,h)\right\},
$$
then $G$ contains a cycle of length at least $k$.
\end{Theorem}

We have sharpness examples $H_{n, k, r}$ and $H_{n, k, h}$.

\begin{Theorem}\label{P}

 Let $G$ be a connected graph on n vertices, with $\delta(G) \geq r\ge 1$, where
$n \geqslant k \geqslant 4$ and $h=\lfloor\frac{k-2}{2}\rfloor$. If
$$
N(K_{s,t},G)>\max \left\{g_{s, t}(n,k-1,r), g_{s, t}(n,k-1,h)\right\},
$$
then $G$ contains a path on $k$ vertices.
\end{Theorem}

This theorem is sharp with the extremal graphs $H_{n, k-1, r}$ and $H_{n, k-1, h}$.
\section{Proofs of Main Results}


P\'{o}sa proved the following result in \cite{Posa62}, which is used to ensure
a cycle of length at least $c$, for any integer $c\ge 3$.

\begin{Lemma}\label{posa} (P\'{o}sa \cite{Posa62})
Let $G$ be a 2-connected $n$-vertex graph and $P$ be a path on $m$
vertices with endpoints $x$ and $y$. Then $G$
contains a cycle of length at least min$\{m, d_P (x) + d_P (y)\}$.
\end{Lemma}

To prove our main results about bipartite graphs, we study bipartite analogue of P\'{o}sa's lemma.
First we need the following lemma.

\begin{Lemma}\label{g-bi} (Jeckson \cite{Ja81})
Let $u$ and $v$ be distinct vertices of a 2-connected graph $G$. Let $P$ be a $uv$-path in $G$ and put
$
T=\left\{w \in N_{G-P}(\{u, v\} \mid N(w) \subseteq V(P)\} .\right.
$
Then there exist internally disjoint uv-paths $P_1$ and $P_2$ such that\\
(a) for $i=1$ and 2, the common vertices of $P_i$ and $P$ occur in the same order in both paths, and\\
(b) $N_P(T \cup\{u, v\}) \subseteq V\left(P_1\right) \cup V\left(P_2\right)$.
\end{Lemma}

Now we show the bipartite analogue of P\'{o}sa's lemma,
which proof is very similar to
Jeckson's proof \cite{Ja85}. 

\begin{Lemma}\label{bi2we}
 Let $G$ be a 2-connected bipartite graph with bipartition $(X, Y)$, and $P$ be a $uv$-path in $G$.\\
(i) If $u \in X$ and $v \in Y$ then $G$ contains a cycle of length at least $\min \{|V(P)|, 2(d_P(u)+d_P(v)-1)\}$.\\
(ii) If $u, v \in X$ then $G$ contains a cycle of length at least $\min \{|V(P)|-1,2(d_P(u)+d_P(v)-2)\}$.
\end{Lemma}
\begin{proof}
(i). Let $P=x_1 y_1 x_2 \cdots x_n y_n$, where $x_1=u$ and $y_n=v$. Let
$S=N_P\left(x_1\right), T=N_P\left(y_n\right)$, and $T^{+}=\left\{y_i \mid x_i \in T\right\}$.
If $S \cap T^{+} \neq \emptyset$, then $G$ contains a cycle of length $|V(P)|$.
Hence we may assume that $S \cap T^{+}=\emptyset$. Suppose there exists $y_i \in S$ and $x_j \in T$ such that $j \leqslant i$.
 Choosing $i$ and $j$ such that $i-j$ is as small a nonnegative integer as possible, it follows that
 $x_{j+1},\cdots,x_i \notin T$ and $y_{j},\cdots,y_{i-1} \notin S$.
 Then
$
C_1=x_1 y_1 x_2 \cdots x_j y_n x_n \cdots y_i x_1
$
is a cycle of $G$ such that $S \cup\left(T^{+} \backslash\left\{y_j\right\}\right) \subseteq V\left(C_1\right) \cap Y$. Thus
$$
\left|V\left(C_1\right)\right| \geqslant 2\left(|S|+\left|T^{+}\right|-1\right)=2(d_P(u)+d_P(v)-1) .
$$
Therefore assume that
\begin{equation}\label{i<j}
\mbox{if }y_i \in S\mbox{ and }x_j \in T, \mbox{ then }i<j.
\end{equation}
Let $P_1$ and $P_2$ be internally disjoint $u v$-paths which satisfy properties (a) and (b) of Lemma \ref{g-bi},
and choose $x_j \in T$. It follows from (b) of Lemma \ref{g-bi}
that $x_j \in P_i$ for some $i \in\{1,2\}$. Let $z_j$ be the successor of $x_j$ along $P_i$ and
put $T^*=\left\{z_j \mid x_j \in T\right\}$. Since $P_1$ and $P_2$ are internally disjoint, $\left|T^*\right| \geqslant|T|-1$.
Moreover, using (a) of Lemma \ref{g-bi} and (\ref{i<j}), it follows that $S \cap T^*=\emptyset$.
Putting $C_2=P_1 \cup P_2$, we have $S \cup T^* \subseteq V\left(C_2\right) \cap Y$ and hence
$$
\left|V\left(C_2\right)\right| \geqslant 2\left(|S|+\left|T^*\right|\right) \geqslant 2(d_P(u)+d_P(v)-1) .
$$

(ii). Putting $P=x_1 y_1 x_2 \cdots y_{n-1} x_n, S=N_P\left(x_1\right), T=N_P\left(x_n\right)$,
and $T^{+}=\left\{y_i \mid y_{i-1} \in T\right\}$, the proof proceeds on similar lines to that of (i).

\end{proof}

Note that if $P$ is a maximal path in $G$ with two end-vertices $u$ and $v$, then $d_G(u)=d_P(u)$ and $d_G(v)=d_P(v)$.
So, we get the following corollary.

\begin{Corollary}\label{bi2} (Jeckson \cite{Ja85}).
 Let $G$ be a 2-connected bipartite graph with bipartition $(X, Y)$, and $P$ be a maximal path in $G$. Let the end vertices of $P$ be $u$ and $v$.\\
(i) If $u \in X$ and $v \in Y$ then $G$ contains a cycle of length at least $\min \{|V(P)|, 2(d(u)+d(v)-1)\}$.\\
(ii) If $u, v \in X$ then $G$ contains a cycle of length at least $\min \{|V(P)|-1,2(d(u)+d(v)-2)\}$.
\end{Corollary}

To prove our results, we also need a definition from Kopylov.

\begin{Definition}
(Kopylov \cite{Ko77}). Let $G$ be a graph and $\alpha$ be a natural number. Delete all
vertices of degree at most $\alpha$ from $G$; for the resulting graph $G_0$, again delete all vertices
of degree at most $\alpha$ from it. We keep running this progress until the minimum degree of
the resulting graph is at least $\alpha + 1$. The resulting graph, denoted by $H(G, \alpha)$, is called
the $(\alpha + 1)$-core of $G$.
\end{Definition}


{\bf Proof of Theorem \ref{CB}:}
Let $h=\lfloor \frac{n-k}{2}\rfloor$.
Let $X,Y$ be the two partition sets of $G$ with $|X|=n\le |Y|=b$.
Suppose for contradiction that
 $G$
 contains no cycle of length $2n-2k$
or more and
\begin{equation}\label{fsr}
N(K_{s,t},G)> \begin{cases}
\max \{f_{s, t}(b,n,n-k,r),f_{s, s}(b,n,n-k,h)\}, & s = t, \\
\max \{f_{s, t}(b,n,n-k,r)+ f_{t, s}(b,n,n-k,r),\\f_{s, t}(b,n,n-k,h)+f_{t, s}(b,n,n-k,h)\}, & s \not= t. \end{cases}
\end{equation}
Let $G_0$ be the bipartite graph obtained from $G$ by adding a dominating vertex
$x$ to $X$ such that $x$ is adjacent to each vertex of $Y$ and adding a dominating vertex
$y$ to $Y$ such that $y$ is adjacent to each vertex of $X$ and adding an edge $xy$.
Then $G_0$ is 2-connected with bipartition $(X', Y')$, has $n+1+b+1$ vertices and contains no cycle of length
$2n-2k+2$ or more, where $X'=X\cup \{x\}$ and $Y'=Y\cup \{y\}$.
Let $G'$ be the $(2n-2k+2)$-closure of $G_0$, i.e., add edges to $G_0$ until any
additional edge creates a cycle of length at least $2n-2k+2$.
Thus $\delta(G')\ge r+1$ and
for any nonadjacent vertices $u\in X'$ and $v\in Y'$ of $G'$, there exists a path
on at least $2n-2k+2$ vertices between $u$ and $v$.
Let $G^*=G'-\{x,y\}$.
Therefore $N(K_{s,t},G^*)\ge N(K_{s,t},G_0-\{x,y\})\ge N(K_{s,t},G)$.
We have that

{\bf Claim 1.}
$H(G',h+1)$ is not empty.

\begin{proof}
Suppose $H(G',h+1)$ is empty.
Note that $x$ is adjacent to each vertex of $Y'$ of $G'$ and $y$ is adjacent to each vertex of $X'$ of $G'$.
For convenience, we divide the proof into the following two cases.

{\bf Case 1.} $s=t$.
In the process of getting $H(G',h+1)$, for the first $n-h$ vertices that are not $x$ in $X'$ of $G'$,
once the $i$-th vertex which has at most $h$ neighbors that are not $y$
is deleted, we delete at most
$\left(\begin{array}{c}
h \\
s
\end{array}\right)\left(\begin{array}{c}
n-i \\
s-1
\end{array}\right)$
copies of $K_{s,s}$, where $i\in [n-h]$;
for the first $b-h$ vertices that are not $y$ in $Y'$ of $G'$,
once the $j$-th vertex which has at most $h$ neighbors that are not $x$ is deleted, we delete at most
$\left(\begin{array}{c}
h \\
s
\end{array}\right)\left(\begin{array}{c}
b-j \\
s-1
\end{array}\right)$
copies of $K_{s,s}$, where $j\in [b-h]$;
 for all the last $2h$ vertices that are not $x,y$ in $G'$, we delete at most
 $\left(\begin{array}{c}
h \\
s
\end{array}\right)\left(\begin{array}{c}
h \\
s
\end{array}\right)$
copies of $K_{s,s}$.
Then we count the number of copies of $K_{s,s}$ as follows.
$$\begin{aligned}
N(K_{s,s},G^*) & \le \sum_{i=1}^{n-h}\left(\begin{array}{c}
h \\
s
\end{array}\right)\left(\begin{array}{c}
n-i \\
s-1
\end{array}\right)
+\sum_{j=1}^{b-h}\left(\begin{array}{c}
h \\
s
\end{array}\right)\left(\begin{array}{c}
b-j \\
s-1
\end{array}\right)+\left(\begin{array}{c}
h \\
s
\end{array}\right)\left(\begin{array}{c}
h \\
s
\end{array}\right)\\&=
\left(\begin{array}{c}
h \\
s
\end{array}\right)
\left(\left(\begin{array}{c}
n \\
s
\end{array}\right)-\left(\begin{array}{c}
h \\
s
\end{array}\right)\right)
+
\left(\begin{array}{c}
h \\
s
\end{array}\right)
\left(\left(\begin{array}{c}
b \\
s
\end{array}\right)-\left(\begin{array}{c}
h \\
s
\end{array}\right)\right)
+\left(\begin{array}{c}
h \\
s
\end{array}\right)\left(\begin{array}{c}
h \\
s
\end{array}\right)
\\&=
\left(\begin{array}{c}
h \\
s
\end{array}\right)
\left(\left(\begin{array}{c}
b \\
s
\end{array}\right)-\left(\begin{array}{c}
h \\
s
\end{array}\right)\right)
+\left(\begin{array}{c}
h \\
s
\end{array}\right)\left(\begin{array}{c}
n \\
s
\end{array}\right)
\\&\le
\left(\begin{array}{c}
n-k-h \\
s
\end{array}\right)\left(\left(\begin{array}{c}
b \\
s
\end{array}\right)-\left(\begin{array}{c}
h \\
s
\end{array}\right)\right)+
\left(\begin{array}{c}
h \\
s
\end{array}\right)
\left(\begin{array}{c}
n \\
s
\end{array}\right)
\\&=f_{s, s}(b,n,n-k,h)
\end{aligned}$$
a contradiction to (\ref{fsr}).

{\bf Case 2.} $s\not=t$.
In the process of getting $H(G',h+1)$,
for the first $n-h$ vertices that are not $x$ in $X'$ of $G'$,
once the $i$-th vertex which has at most $h$ neighbors that are not $y$ is deleted, we delete at most
$\left(\begin{array}{c}
h \\
s
\end{array}\right)\left(\begin{array}{c}
n-i \\
t-1
\end{array}\right)+
\left(\begin{array}{c}
h \\
t
\end{array}\right)\left(\begin{array}{c}
n-i \\
s-1
\end{array}\right)$
copies of $K_{s,t}$, where $i\in [n-h]$;
for the first $b-h$ vertices that are not $y$ in $Y'$ of $G'$,
once the $j$-th vertex which has at most $h$ neighbors that are not $x$ is deleted, we delete at most
 $\left(\begin{array}{c}
h \\
s
\end{array}\right)\left(\begin{array}{c}
b-j \\
t-1
\end{array}\right)+
\left(\begin{array}{c}
h \\
t
\end{array}\right)\left(\begin{array}{c}
b-j \\
s-1
\end{array}\right)$
copies of $K_{s,t}$, where $j\in [b-h]$;
 for all the last $2h$ vertices that are not $x,y$ of $G'$,
we delete at most
 $\left(\begin{array}{c}
h \\
s
\end{array}\right)\left(\begin{array}{c}
h \\
t
\end{array}\right)+
\left(\begin{array}{c}
h \\
t
\end{array}\right)\left(\begin{array}{c}
h \\
s
\end{array}\right)$
copies of $K_{s,t}$.
Then the
 number of copies of $K_{s,t}$ can be estimated as follows.

$$\begin{aligned}
N(K_{s,t},G^*)\le &\sum_{i=1}^{n-h}\left(
\left(\begin{array}{c}
h \\
s
\end{array}\right)\left(\begin{array}{c}
n-i \\
t-1
\end{array}\right)+
\left(\begin{array}{c}
h \\
t
\end{array}\right)\left(\begin{array}{c}
n-i \\
s-1
\end{array}\right)
\right)
+\\&\sum_{j=1}^{b-h}\left(
\left(\begin{array}{c}
h \\
s
\end{array}\right)\left(\begin{array}{c}
b-j \\
t-1
\end{array}\right)+
\left(\begin{array}{c}
h \\
t
\end{array}\right)\left(\begin{array}{c}
b-j \\
s-1
\end{array}\right)
\right)
+2\left(\begin{array}{c}
h \\
s
\end{array}\right)\left(\begin{array}{c}
h \\
t
\end{array}\right)
\\=&
\left(\begin{array}{c}
h \\
s
\end{array}\right)
\left(\left(\begin{array}{c}
n \\
t
\end{array}\right)-\left(\begin{array}{c}
h \\
t
\end{array}\right)\right)
+
\left(\begin{array}{c}
h \\
t
\end{array}\right)
\left(\left(\begin{array}{c}
b \\
s
\end{array}\right)-\left(\begin{array}{c}
h \\
s
\end{array}\right)\right)
+
\left(\begin{array}{c}
h \\
s
\end{array}\right)\left(\begin{array}{c}
h \\
t
\end{array}\right)+
\\&
\left(\begin{array}{c}
h \\
t
\end{array}\right)
\left(\left(\begin{array}{c}
n \\
s
\end{array}\right)-\left(\begin{array}{c}
h \\
s
\end{array}\right)\right)
+
\left(\begin{array}{c}
h \\
s
\end{array}\right)
\left(\left(\begin{array}{c}
b \\
t
\end{array}\right)-\left(\begin{array}{c}
h \\
t
\end{array}\right)\right)
+
\left(\begin{array}{c}
h \\
s
\end{array}\right)\left(\begin{array}{c}
h \\
t
\end{array}\right)
\\=&
\left(\begin{array}{c}
h \\
t
\end{array}\right)
\left(\left(\begin{array}{c}
b \\
s
\end{array}\right)-\left(\begin{array}{c}
h \\
s
\end{array}\right)\right)
+
\left(\begin{array}{c}
h \\
s
\end{array}\right)\left(\begin{array}{c}
n \\
t
\end{array}\right)
+\\&
\left(\begin{array}{c}
h \\
s
\end{array}\right)
\left(\left(\begin{array}{c}
b \\
t
\end{array}\right)-\left(\begin{array}{c}
h \\
t
\end{array}\right)\right)
+
\left(\begin{array}{c}
h \\
t
\end{array}\right)\left(\begin{array}{c}
n \\
s
\end{array}\right)
\\ \le &
\left(\begin{array}{c}
n-k-h \\
t
\end{array}\right)
\left(\left(\begin{array}{c}
b \\
s
\end{array}\right)-\left(\begin{array}{c}
h \\
s
\end{array}\right)\right)
+
\left(\begin{array}{c}
h \\
s
\end{array}\right)\left(\begin{array}{c}
n \\
t
\end{array}\right)+
\\&
\left(\begin{array}{c}
n-k-h \\
s
\end{array}\right)
\left(\left(\begin{array}{c}
b \\
t
\end{array}\right)-\left(\begin{array}{c}
h \\
t
\end{array}\right)\right)
+
\left(\begin{array}{c}
h \\
t
\end{array}\right)\left(\begin{array}{c}
n \\
s
\end{array}\right)
\\=&f_{s, t}(b,n,n-k,h)+ f_{t, s}(b,n,n-k,h)
\end{aligned}$$
a contradiction to (\ref{fsr}).

\end{proof}
{\bf Claim 2.}
$H(G',h+1)$ is a complete bipartite graph.

\begin{proof}
If there exist two nonadjacent vertices $u$ and $v$ in $H(G',h+1)$ and $u\in X',v\in Y'$, then
there is a path on at least $2n- 2k + 2$ vertices between $u$ and $v$.
Among all these nonadjacent pairs of
vertices in $H(G',h+1)$, we choose $u, v\in H(G',h+1)$ with $u\in X',v\in Y'$ such that the path between them is the
longest.
If all neighbors of $u$ in $H(G',h+1)$ lie in $P$ and all neighbors of $v$ in $H(G',h+1)$ lie in $P$,
then by (i) of Lemma \ref{bi2we}, $G'$ has a cycle of length at least min$\{2n-2k+2, 2(d_P(u)+d_P(v)-1)\}$=
min$\{2n-2k+2, 2(h+2+h+2-1)\}=2n-2k+2$, a contradiction.
If there exists a neighbor $x$ of $u$ in $H(G',h+1)$ not lie in $P$
and there exists a neighbor $y$ of $v$ in $H(G',h+1)$ not lie in $P$,
then by the maximality of $P$,
$xy\in E(H(G',h+1))$. Then $G'$ has a cycle of length at least $2n-2k+4$, a contradiction.
Hence either there exists a neighbor of $u$ in $H(G',h+1)$ not lie in $P$ and all neighbors of $v$ in $H(G',h+1)$ lie in $P$,
or there exists a neighbor of $v$ in $H(G',h+1)$ not lie in $P$ and all neighbors of $u$ in $H(G',h+1)$ lie in $P$.
W.l.o.g., suppose that
there exists a neighbor $x$ of $u$ in $H(G',h+1)$ not lie in $P$
and all neighbors of $v$ in $H(G',h+1)$ lie in $P$.
If there exists a neighbor $x'$ of $x$ in $H(G',h+1)$ not lie in $P$, then
by the maximality of $P$,
$x'v\in E(H(G',h+1))$. Then $G'$ has a cycle of length at least $2n-2k+4$, a contradiction.
Thus all neighbors of $x$ in $H(G',h+1)$ lie in $P$
and there is a path on at least $2n- 2k + 3$ vertices between $x$ and $v$, and $x,v$ in same part.
By (ii) of Lemma \ref{bi2we}, $G'$ has a cycle of length at least min$\{2n-2k+3-1, 2(d_P(x)+d_P(v)-2)\}$=
min$\{2n-2k+2, 2(h+2+h+2-2)\}=2n-2k+2$, a contradiction.

\end{proof}


{\bf Claim 3.} Let $l+1=$min$\{|V(H(G',h+1))\cap X'|,|V(H(G',h+1))\cap Y'|\}$.
Then
$r\le n-k-l \le h$.

\begin{proof}
By the definition of $H(G',h+1)$, $l+1\ge h+2$.
Thus $n-k-l \le h$.
Next we show that $l\le n-k-r$.
Suppose $l+1\ge n-k-r+2$.
If $u\in V(G'-H(G',h+1))$, then $u$ is not adjacent to at least one vertex in $ H(G',h+1)$
such that they are in different parts.
Otherwise, $u\in H(G',h+1)$.
We pick $u\in V(G'-H(G',h+1))$ and $v\in V(H(G',h+1))$ satisfying the following three conditions:
(a) $u,v$ in different parts, (b) $u$ and $v$ are not adjacent, and (c) a longest path in $G'$ from $u$ to $v$
contains the largest number of edges among such nonadjacent pairs.
Let $P$ be a longest path in $G'$ from $u$ to $v$.
Clearly, $|V(P)|\ge 2n-2k+2$ as $G'$ is edge maximal.
If all neighbors of $u$ in $G'-H(G',h+1)$ lie in $P$ and all neighbors of $v$ in $H(G',h+1)$ lie in $P$,
then by (i) of Lemma \ref{bi2we}, $G'$ has a cycle of length at least min$\{2n-2k+2, 2(d_P(u)+d_P(v)-1)\}$=
min$\{2n-2k+2, 2(r+1+n-k-r+2-1)\}=2n-2k+2$, a contradiction.
If there exists a neighbor $x$ of $u$ in $G'-H(G',h+1)$ not lie in $P$
and there exists a neighbor $y$ of $v$ in $H(G',h+1)$ not lie in $P$,
then by the maximality of $P$,
$xy\in E(G')$. Then $G'$ has a cycle of length at least $2n-2k+4$, a contradiction.
Hence either there exists a neighbor $x$ of $u$ in $G'-H(G',h+1)$ not lie in $P$ and all neighbors of $v$ in $H(G',h+1)$ lie in $P$,
or there exists a neighbor $y$ of $v$ in $H(G',h+1)$ not lie in $P$ and all neighbors of $u$ in $G'-H(G',h+1)$ lie in $P$.
W.l.o.g., suppose that
there exists a neighbor $x$ of $u$ in $G'-H(G',h+1)$ not lie in $P$
and all neighbors of $v$ in $H(G',h+1)$ lie in $P$.
If there exists a neighbor $x'$ of $x$ in $H(G',h+1)$ not lie in $P$, then
by the maximality of $P$,
$x'v\in E(H(G',h+1))$. Then $G'$ has a cycle of length at least $2n-2k+4$, a contradiction.
Thus all neighbors of $x$ in $H(G',h+1)$ lie in $P$
and
there is a path on at least $2n- 2k + 3$ vertices between $x$ and $v$, and $x,v$ in same part.
By (ii) of Lemma \ref{bi2we}, $G'$ has a cycle of length at least min$\{2n-2k+3-1, 2(d_P(x)+d_P(v)-2)\}$=
min$\{2n-2k+2, 2(r+1+n-k-r+2-2)\}=2n-2k+2$, a contradiction.
Thus $l+1\le n+1-k-r$.
Therefore $r\le n-k-l \le h$.

\end{proof}

{\bf Claim 4.}
$H(G',h+1)\not= H(G',n+1-k-l)$.

\begin{proof}
Suppose $H(G',h+1)= H(G',n+1-k-l)$.
Note that $x,y$ must be contained in $H(G',h+1)$ as
$x$ is adjacent to each vertex of $Y'$ of $G'$ and $y$ is adjacent to each vertex of $X'$ of $G'$.
We divide the proof into the following two cases:

{\bf Case 1.}
$l+1=|V(H(G',h+1))\cap X'|$.

{\bf Subcase 1.1.}
$s=t$.
Firstly, in $H(G',h+1)=H(G',n+1-k-l)$,
the number of copies of $K_{s,s}$ that do not include $x,y$ is at most
$\left(\begin{array}{c}
l \\
s
\end{array}\right)\left(\begin{array}{c}
b \\
s
\end{array}\right)$.
Secondly, in $G'-H(G',n+1-k-l)$, every vertex had at most $n-k-l$ neighbors that were not $y$ at the time of its deletion.
We have
$$\begin{aligned}
N(K_{s,s},G^*)&\le
\left(\begin{array}{c}
l \\
s
\end{array}\right)\left(\begin{array}{c}
b \\
s
\end{array}\right)+
\sum_{i=1}^{n-l}\left(\begin{array}{c}
n-k-l \\
s
\end{array}\right)\left(\begin{array}{c}
n-i \\
s-1
\end{array}\right)
\\&=
\left(\begin{array}{c}
n-k-l \\
s
\end{array}\right)
\left(\left(\begin{array}{c}
n \\
s
\end{array}\right)-\left(\begin{array}{c}
l \\
s
\end{array}\right)\right)
+\left(\begin{array}{c}
l \\
s
\end{array}\right)\left(\begin{array}{c}
b \\
s
\end{array}\right)
\\&=f_{s, s}(b,n,n-k,n-k-l)
\\&\le\max \{f_{s, s}(b,n,n-k,r),f_{s, s}(b,n,n-k,h)\}\end{aligned}$$
a contradiction to (\ref{fsr}).

{\bf Subcase 1.2.} $s\not=t$.
Firstly, in $H(G',h+1)=H(G',n+1-k-l)$,
the number of copies of $K_{s,s}$ that do not include $x,y$ is at most
 $\left(\begin{array}{c}
l \\
s
\end{array}\right)\left(\begin{array}{c}
b \\
t
\end{array}\right)+
\left(\begin{array}{c}
l \\
t
\end{array}\right)\left(\begin{array}{c}
b \\
s
\end{array}\right)$.
Secondly, in $G'-H(G',n+1-k-l)$, every vertex had at most $n-k-l$ neighbors that were not $y$ at the time of its deletion.
We have
$$\begin{aligned}
N(K_{s,s},G^*)\le & \sum_{i=1}^{n-l}
\left(\begin{array}{c}
n-k-l \\
s
\end{array}\right)\left(\begin{array}{c}
n-i \\
t-1
\end{array}\right)+
\sum_{i=1}^{n-l}\left(\begin{array}{c}
n-k-l \\
t
\end{array}\right)\left(\begin{array}{c}
n-i \\
s-1
\end{array}\right)+
\\&
\left(\begin{array}{c}
l \\
s
\end{array}\right)\left(\begin{array}{c}
b \\
t
\end{array}\right)+
\left(\begin{array}{c}
l \\
t
\end{array}\right)\left(\begin{array}{c}
b \\
s
\end{array}\right)
\\=&
\left(\begin{array}{c}
n-k-l \\
s
\end{array}\right)
\left(\left(\begin{array}{c}
n \\
t
\end{array}\right)-\left(\begin{array}{c}
l \\
t
\end{array}\right)\right)
+
\left(\begin{array}{c}
l \\
t
\end{array}\right)
\left(\begin{array}{c}
b \\
s
\end{array}\right)+
\\&
\left(\begin{array}{c}
n-k-l \\
t
\end{array}\right)
\left(\left(\begin{array}{c}
n \\
s
\end{array}\right)-\left(\begin{array}{c}
l \\
s
\end{array}\right)\right)
+
\left(\begin{array}{c}
l \\
s
\end{array}\right)
\left(\begin{array}{c}
b \\
t
\end{array}\right)
\\=&f_{s, t}(b,n,n-k,n-k-l)+ f_{t, s}(b,n,n-k,n-k-l)
\\ \le &\max \{f_{s, t}(b,n,n-k,r)+ f_{t, s}(b,n,n-k,r),f_{s, t}(b,n,n-k,h)+
\\&f_{t, s}(b,n,n-k,h)\}
\end{aligned}$$
a contradiction to (\ref{fsr}).

{\bf Case 2.}
$l+1=|V(H(G',h+1))\cap Y'|$.

{\bf Subcase 2.1.}
$s=t$.
Firstly, in $H(G',h+1)=H(G',n+1-k-l)$,
the number of copies of $K_{s,s}$ that do not include $x,y$ is at most
$\left(\begin{array}{c}
l \\
s
\end{array}\right)\left(\begin{array}{c}
n \\
s
\end{array}\right)$.
Secondly, in $G'-H(G',n+1-k-l)$, every vertex had at most $n-k-l$ neighbors that were not $x$ at the time of its deletion.
We have
$$\begin{aligned}
N(K_{s,s},G^*)&\le \sum_{i=1}^{b-l}\left(\begin{array}{c}
n-k-l \\
s
\end{array}\right)\left(\begin{array}{c}
b-i \\
s-1
\end{array}\right)
+\left(\begin{array}{c}
l \\
s
\end{array}\right)\left(\begin{array}{c}
n \\
s
\end{array}\right)
\\&=
\left(\begin{array}{c}
n-k-l \\
s
\end{array}\right)
\left(\left(\begin{array}{c}
b \\
s
\end{array}\right)-\left(\begin{array}{c}
l \\
s
\end{array}\right)\right)
+\left(\begin{array}{c}
l \\
s
\end{array}\right)\left(\begin{array}{c}
n \\
s
\end{array}\right)
\\&\le
\left(\begin{array}{c}
n-k-l \\
s
\end{array}\right)
\left(\left(\begin{array}{c}
n \\
s
\end{array}\right)-\left(\begin{array}{c}
l \\
s
\end{array}\right)\right)
+\left(\begin{array}{c}
l \\
s
\end{array}\right)\left(\begin{array}{c}
b \\
s
\end{array}\right)
\\&=f_{s, s}(b,n,n-k,n-k-l)
\\&\le\max \{f_{s, s}(b,n,n-k,r),f_{s, s}(b,n,n-k,h)\}
\end{aligned}$$
a contradiction to (\ref{fsr}).

{\bf Subcase 2.2.} $s\not=t$.
Firstly, in $H(G',h+1)=H(G',n+1-k-l)$,
the number of copies of $K_{s,s}$ that do not include $x,y$ is at most
 $\left(\begin{array}{c}
l \\
s
\end{array}\right)\left(\begin{array}{c}
n \\
t
\end{array}\right)+
\left(\begin{array}{c}
l \\
t
\end{array}\right)\left(\begin{array}{c}
n \\
s
\end{array}\right)$.
Secondly, in $G'-H(G',n+1-k-l)$, every vertex had at most $n-k-l$ neighbors that were not $x$ at the time of its deletion.
We have

$$\begin{aligned}
N(K_{s,s},G^*)\le & \sum_{i=1}^{b-l}
\left(\begin{array}{c}
n-k-l \\
s
\end{array}\right)\left(\begin{array}{c}
b-i \\
t-1
\end{array}\right)+
\sum_{i=1}^{b-l}\left(\begin{array}{c}
n-k-l \\
t
\end{array}\right)\left(\begin{array}{c}
b-i \\
s-1
\end{array}\right)+
\\&\left(\begin{array}{c}
l \\
s
\end{array}\right)\left(\begin{array}{c}
n \\
t
\end{array}\right)+
\left(\begin{array}{c}
l \\
t
\end{array}\right)\left(\begin{array}{c}
n \\
s
\end{array}\right)
\\=&
\left(\begin{array}{c}
n-k-l \\
s
\end{array}\right)
\left(\left(\begin{array}{c}
b \\
t
\end{array}\right)-\left(\begin{array}{c}
l \\
t
\end{array}\right)\right)
+
\left(\begin{array}{c}
l \\
t
\end{array}\right)
\left(\begin{array}{c}
n \\
s
\end{array}\right)+
\\&
\left(\begin{array}{c}
n-k-l \\
t
\end{array}\right)
\left(\left(\begin{array}{c}
b \\
s
\end{array}\right)-\left(\begin{array}{c}
l \\
s
\end{array}\right)\right)
+
\left(\begin{array}{c}
l \\
s
\end{array}\right)
\left(\begin{array}{c}
n \\
t
\end{array}\right)
\\ \le &
\left(\begin{array}{c}
n-k-l \\
s
\end{array}\right)
\left(\left(\begin{array}{c}
n \\
t
\end{array}\right)-\left(\begin{array}{c}
l \\
t
\end{array}\right)\right)
+
\left(\begin{array}{c}
l \\
t
\end{array}\right)
\left(\begin{array}{c}
b \\
s
\end{array}\right)+
\\&
\left(\begin{array}{c}
n-k-l \\
t
\end{array}\right)
\left(\left(\begin{array}{c}
n \\
s
\end{array}\right)-\left(\begin{array}{c}
l \\
s
\end{array}\right)\right)
+
\left(\begin{array}{c}
l \\
s
\end{array}\right)
\left(\begin{array}{c}
b \\
t
\end{array}\right)
\\=&f_{s, t}(b,n,n-k,n-k-l)+ f_{t, s}(b,n,n-k,n-k-l)
\\
\le & \max \{f_{s, t}(b,n,n-k,r)+ f_{t, s}(b,n,n-k,r),f_{s, t}(b,n,n-k,h)+
\\&f_{t, s}(b,n,n-k,h)\}
\end{aligned}$$
a contradiction to (\ref{fsr}).

\end{proof}

{\bf Claim 5.}
$G'$ contains a cycle of length at least $2n-2k+2$.

\begin{proof}
Note that $H(G',h+1)\subseteq H(G',n+1-k-l)$.
By Claim 4, $H(G',h+1)$ is a proper subgraph of $H(G',n+1-k-l)$ and there must be a vertex $u$ in
$H(G',h+1)$ and a vertex $v$ in $H(G',n+1-k-l)$ that are nonadjacent and $u,v$ in different parts.
 Among all such pairs of vertices, we choose $u\in V (H(G', h+1))$ and $v \in
V (H(G', n+1-k-l))$ with $u,v$ in different parts such that there is a longest path $P$ between them. Then $P$ contains at
least $2n-2k+2$ vertices.

If all neighbors of $u$ in $H(G',h+1)$ lie in $P$ and all neighbors of $v$ in $H(G',n+1-k-l)$ lie in $P$,
then by (i) of Lemma \ref{bi2we}, $G'$ has a cycle of length at least min$\{2n-2k+2, 2(d_P(u)+d_P(v)-1)\}$=
min$\{2n-2k+2, 2(l+1+n+1-k-l+1-1)\}=2n-2k+2$, a contradiction.
If there exists a neighbor $x$ of $u$ in $H(G',h+1)$ not lie in $P$
and there exists a neighbor $y$ of $v$ in $H(G',n+1-k-l)$ not lie in $P$,
then by the maximality of $P$,
$xy\in E(H(G',h+1))$. Then $G'$ has a cycle of length at least $2n-2k+4$, a contradiction.
Hence either there exists a neighbor $x$ of $u$ in $H(G',h+1)$ not lie in $P$ and all neighbors of $v$ in $H(G',n+1-k-l)$ lie in $P$,
or there exists a neighbor $y$ of $v$ in $H(G',n+1-k-l)$ not lie in $P$ and all neighbors of $u$ in $H(G',h+1)$ lie in $P$.
W.l.o.g., suppose that
there exists a neighbor $x$ of $u$ in $H(G',h+1)$ not lie in $P$
and all neighbors of $v$ in $H(G',n+1-k-l)$ lie in $P$.
If there exists a neighbor $x'$ of $x$ in $H(G',h+1)$ not lie in $P$, then
by the maximality of $P$,
$x'v\in E(H(G',h+1))$. Then $G'$ has a cycle of length at least $2n-2k+4$, a contradiction.
Thus all neighbors of $x$ in $H(G',h+1)$ lie in $P$
and
there is a path on at least $2n- 2k + 3$ vertices between $x$ and $v$, and $x,v$ in same part.
By (ii) of Lemma \ref{bi2we}, $G'$ has a cycle of length at least min$\{2n-2k+3-1, 2(d_P(x)+d_P(v)-2)\}$=
min$\{2n-2k+2, 2(l+1+n+1-k-l+1-2)\}=2n-2k+2$, a contradiction.
\end{proof}

Claim 5 contradicts our assumption. Hence
the proof of Theorem \ref{CB} is complete.
\\
\\
{\bf Proof of Theorem \ref{PB}:}
Let $h=\lfloor \frac{n-k-1}{2}\rfloor$.
Let $X,Y$ be the two partition sets of $G$ with $|X|=n\le |Y|=b$.
Suppose for contradiction that
 $G$
 contains no path on $2n-2k$ vertices and
\begin{equation}\label{fsr1}
N(K_{s,t},G)> \begin{cases}
\max \{f_{s, t}(b,n,n-k-1,r),f_{s, s}(b,n,n-k-1,h)\}, & s = t, \\
\max \{f_{s, t}(b,n,n-k-1,r)+ f_{t, s}(b,n,n-k-1,r),\\f_{s, t}(b,n,n-k-1,h)+f_{t, s}(b,n,n-k-1,h)\}, & s \not= t. \end{cases}
\end{equation}
Then
 $G$ contains no cycle of length $2n-2k$ or more.
 If $G$ contains a cycle $C$ of length $2n-2k-2$,
 then $|V(C)\cup X|= n-k-1$ and $|V(C)\cup Y|= n-k-1$.
 Since $G$ is connected, we can get a path on $2n-2k$ vertices, a contradiction.
 Thus $G$ contains no cycle of length $2n-2k-2$ or more.
Then by Theorem \ref{CB},\\
 $$N(K_{s,t},G)\le \begin{cases}
\max \{f_{s, t}(b,n,n-k-1,r),f_{s, s}(b,n,n-k-1,h)\}, & s = t, \\
\max \{f_{s, t}(b,n,n-k-1,r)+ f_{t, s}(b,n,n-k-1,r),\\f_{s, t}(b,n,n-k-1,h)+f_{t, s}(b,n,n-k-1,h)\}, & s \not= t, \end{cases}$$
a contradiction to (\ref{fsr1}).
\\

Next we give very short proof of Theorem \ref{C}.
Our proof is very similar to the proof of Theorem  in \cite{LYZ21}.
We give only a sketch and omit the details.
Similar to Theorem \ref{C}, we have Theorem \ref{P} and skip the details of the proof.

{\bf Proof of Theorem \ref{C}:}
Let $n \geqslant k \geqslant 5$ and $h=\lfloor\frac{k-1}{2}\rfloor$. Let $G$ be an edge-maximal counter-example, i.e., adding any additional edge to $G$ creates a cycle of length at least $k$ and
\begin{equation}\label{gsr}
N(K_{s,t},G)>\max \left\{g_{s, t}(n,k,r), g_{s, t}(n,k,h)\right\}.
\end{equation}
Thus for each pair of nonadjacent vertices $u$ and $v$ of $G$,
there exists a path on at least $k$ vertices between $u$ and $v$. We have that

{\bf Claim 1} (\cite{LYZ21}). $H(G, h)$ is not empty.

{\bf Claim 2} (\cite{LYZ21}). $H(G, h)$ is a clique.

%
The main differences come from Claims 3 and 4, whose proofs need the minimum degree
condition and a new function.

{\bf Claim 3.} Let $\ell=|V(H(G, h))|$. Then $r \leqslant k-\ell \leqslant h$.

\begin{proof} Since each vertex of $H(G, h)$ has degree at least $h+1$, we have $\ell \geqslant h+2$.
Thus $k-\ell \le h$.
Next we show that $\ell\le k-r$, where $\delta(G)\ge r$.
Suppose $\ell\ge k-r+1$.
If $x\in V(G)\backslash V(H(G,h))$, then $x$ is not adjacent to at least one vertex in $ H(G,h)$.
Otherwise, $x\in H(G,h)$.
We pick $x\in V(G)\backslash V(H(G,h))$ and $y\in V(H(G,h))$ satisfying the following two conditions:
(a) $x$ and $y$ are not adjacent, and (b) a longest path in $G$ from $x$ to $y$
contains the largest number of edges among such nonadjacent pairs.
Let $P$ be a longest path in $G$ from $x$ to $y$.
 Then by the maximality of $P$, all neighbors of $x$ in $H(G, h)$ lie in $P$:
 if $x$ has a neighbor $x'\in H(G,h)-P$, then either $x'y\in E(G)$ and $x'P$ is a cycle of length at least $k$,
 or $x'y\notin E(G)$ and so $x'P$ is a longer path.
 Similarly, all neighbors of $y$ in $H(G, h)$ lie in $P$.
Then
by Lemma \ref{posa}, $G$ has a cycle of length at least min$\{k, d_P(x)+d_P(y))\}$=
min$\{k, r+k-r)=k$, a contradiction.
Thus $\ell\le k-r$.
Therefore $r\le k-\ell \le h$.

\end{proof}

{\bf Claim 4.}
$H(G,h)\not= H(G,k-\ell)$.

\begin{proof}
Suppose $H(G,h)= H(G,k-\ell)$.
We divide the proof into the following two cases:

{\bf Case 1.} $s=t$.
Firstly, the number of copies of $K_{s,s}$ contained in $H(G,h)=H(G,k-\ell)$
is at most $\frac{1}{2}\left(\begin{array}{c}
\ell \\
2 s
\end{array}\right)\left(\begin{array}{c}
2 s \\
s
\end{array}\right)$.
Secondly, in $G-H(G,k-\ell)$, every vertex had at most $k-\ell$ neighbors at the time of its deletion.
Therefore
$$
\begin{aligned}
N\left(K_{s, s}, G\right) & \leqslant \sum_{i=1}^{n-\ell}\left(\begin{array}{c}
k-\ell \\
s
\end{array}\right)\left(\begin{array}{c}
n-s-i \\
s-1
\end{array}\right)+\frac{1}{2}\left(\begin{array}{c}
\ell \\
2 s
\end{array}\right)\left(\begin{array}{c}
2 s \\
s
\end{array}\right) \\
& =g_{s, s}(n,k,k-\ell)\\
& \le \max \left\{g_{s, s}(n,k,r), g_{s, s}(n,k,h)\right\},
\end{aligned}
$$
a contradiction to (\ref{gsr}). Thus $H(G,h)\not= H(G,k-\ell)$.

{\bf Case 2.} $s\not=t$. Similarly,
we count the number of copies of $K_{s,t}$ as follows:
$$
\begin{aligned}
N\left(K_{s, t}, G\right)  \leqslant & \sum_{i=1}^{n-\ell}\left(\left(\begin{array}{c}
k-\ell \\
s
\end{array}\right)\left(\begin{array}{c}
n-s-i \\
t-1
\end{array}\right)+\left(\begin{array}{c}
k-\ell \\
t
\end{array}\right)\left(\begin{array}{c}
n-t-i \\
s-1
\end{array}\right)\right)+
\\&
\left(\begin{array}{c}
\ell \\
s+t
\end{array}\right)\left(\begin{array}{c}
s+t \\
s
\end{array}\right) \\
=&g_{s, t}(n,k,k-\ell)\\
 \le &\max \left\{g_{s, t}(n,k,r), g_{s, t}(n,k,h)\right\}
\end{aligned}
$$
a contradiction to (1). Hence $H(G,h)\not= H(G,k-\ell)$.
\end{proof}

{\bf Claim 5 }(\cite{LYZ21}).
$G$ contains a cycle of length at least $k$.

Claim 5 contradicts our assumption. The proof of
Theorem \ref{C} is complete.
%
%


\section{Future Work}
It is natural to raise the following conjecture.
\begin{Conjecture}\label{ConjCB}
Let $G$ be a connected bipartite graph with bipartition $(X, Y)$.
Suppose $|X|=n\le |Y|=b,h=\lfloor\frac{n-k}{2}\rfloor$ and $\delta(G) \geq r\ge 1$, where $n\ge 2k+2r$ and $k\in \mathbb{Z}$.
$$
N(K_{s,t},G)> \begin{cases}
\max \{f_{s, t}(b,n,n-k,r),f_{s, s}(b,n,n-k,h)\}, & s = t, \\
\max \{f_{s, t}(b,n,n-k,r)+ f_{t, s}(b,n,n-k,r),\\f_{s, t}(b,n,n-k,h)+
f_{t, s}(b,n,n-k,h)\}, & s \not= t. \end{cases}
$$
then $G$ contains a cycle of length $2n-2k$.
\end{Conjecture}

Recently, Li and Ning \cite{LN16}, and independently, F\"{u}redi, Kostochka and Luo \cite{FKL17} proved a stability
version of Theorem \ref{MM63}.
It would be interesting, though, to prove a stability
version of Theorem \ref{CB} or Conjecture \ref{ConjCB}, if it is true.



\begin{thebibliography}{19}

\bibitem{AS16}
N. Alon and C. Shikhelman, Many $T$ copies in $H$-free graphs, J. Combin. Theory
Ser. B 121 (2016), 146-172.

\bibitem{AA09}
J. Adamus and L. Adamus, Ore and Erd\H{o}s type conditions for long cycles
in balanced bipartite graphs, Discrete Math. Theor. Comput. Sci. 11:2 (2009), 57-70.



\bibitem{EG59}
P. Erd\H{o}s and T. Gallai, On maximal paths and circuits of graphs, Acta Math. Hungar. 10(3-4) (1959), 337-356.


\bibitem{Er62}
P. Erd\H{o}s, Remarks on a paper of P\'{o}sa, Magy. Tud. Akad. Mat. Kut. Int\'{e}z. K\"{o}z 7 (1962), 227-229.

\bibitem{GG20}
D. Gerbner and E. Gy\H{o}ri, A. Methuku and M. Vizer. Generalized Tu\'{r}an numbers for
even cycles, J. Combin. Theory Ser. B 145 (2020), 169-213.

\bibitem{GMV19}
D. Gerbner, A. Methuku, and M. Vizer, Generalized Tu\'{r}an problems for disjoint
copies of graphs, Discrete Math. 342(11) (2019), 3130-3141.

\bibitem{GP19}
D. Gerbner and C. Palmer, Counting copies of a fixed subgraph in $F$-free graphs,
European J. Combin. 82 (2019), 103001.

\bibitem{GS20}
L. Gishboliner and A. Shapira, A generalized Tur\'{r}n problem and its applications,
Int. Math. Res. Not. 2020(11) (2020), 3417-3452.

\bibitem{GL12}
E. Gy\H{o}ri and H. Li, The maximum number of triangles in $C_{2k+1}$-free graph, J.
Combin. Theory Ser. B 102 (2012), 1061-1066.

\bibitem{FKL17}
Z. F\"{u}redi, A. Kostochka and R. Luo, A stability version for a theorem of Erd\H{o}s
on nonhamiltonian graphs, Discrete Math. 340 (2017), 2688-2690.




\bibitem{Ja81}
B. Jackson, Cycles in bipartite graphs, J. Combin. Theory Ser. B 30 (1981), 332-342.

\bibitem{Ja85}
 B. Jackson, Long cycles in bipartite graphs, J. Combin. Theory, Ser. B 38 (1985), 118-131.



\bibitem{Ko77}
G. N. Kopylov, Maximal paths and cycles in a graph, Dokl. Akad. Nauk SSSR 234 19-21. English translation:
Soviet Math. Dokl. 18 (1977), 593-596.





\bibitem{LN21}
B. Li and B. Ning, Exact bipartite Tur\'{a}n numbers of large even
cycles, J. Graph Theory 97 (2021), 642-656.

\bibitem{LN16}
B. Li and B. Ning: Spectral analogues of Erd\H{o}s' and Moon-Moser's theorems on
Hamilton cycles, Linear Multilinear Algebra 64 (2016), 2252-2269.

\bibitem{Luo17}
R. Luo, The maximum number of cliques in graphs without long cycles, J. Combin.
Theory Ser. B 128 (2017), 219-226.

\bibitem{LYZ21}
C.H. Lu, L.T. Yuan and P. Zhang, The Maximum Number of Copies of $K_{r,s}$ in Graphs Without Long Cycles or Paths,
Electron. J. Combin. 28(4) (2021).

\bibitem{MM63}
J. Moon and L. Moser, On Hamiltonian bipartite graphs, Israel J. Math. 1 (3) (1963) 163-165.

\bibitem{MQ20}
J. Ma and Y. Qiu, Some sharp results on the generalized Tu\'{r}an numbers, European
J. Combin. 84 (2020) 103026.

\bibitem{NP20}
B. Ning and X. Peng, Extensions of the Erd\H{o}s-Gallai theorem
and Luo's theorem, Comb. Probab. Comput. 29 (2020), 128-136.

\bibitem{Ore61}
O. Ore, Arc coverings of graphs, Ann. Mat. Pura Appl. 55 (1961), 315-321.

\bibitem{Posa62}
L. P\'{o}sa, A theorem concerning Hamilton lines, Magyar Tud. Akad. Mat. Kutat\'{o} Int.
K\"{o}zl 7 (1962), 225-226.


\bibitem{Wang20}
J. Wang, The shifting method and generalized Tur\'{a}n
number of matchings, European
J. Combin. 85 (2020), 103057.



\end{thebibliography}
\end{document}